\documentclass[a4paper,10pt, twoside]{article}
\usepackage{url, color, epsfig, amsmath, amsthm, amssymb}
\usepackage[top=2cm, bottom=3cm, left=2.5cm, right=2.5cm]{geometry}
\usepackage{algorithm}
\usepackage[noend]{algorithmic}
\usepackage{amsmath}
\usepackage{mathtools}
\usepackage{enumerate}
\usepackage{colortbl}
\usepackage[textwidth=1.850cm, textsize=scriptsize]{todonotes}
\usepackage{afterpage}
\usepackage{multicol}
\usepackage{multirow}
\usepackage{caption}
\usepackage{wrapfig}
\usepackage{thmtools}
\usepackage{thm-restate}
\usepackage{fancyhdr}
\usepackage{subcaption}
\usepackage{enumitem}

\allowdisplaybreaks

\newtheorem{theorem}{Theorem}
\newtheorem*{theorem*}{Theorem}
\newtheorem{corollary}[theorem]{Corollary}
\newtheorem{lemma}[theorem]{Lemma}

\newtheorem*{definition*}{Definition}

\newtheorem*{example*}{Example}

\newtheorem*{claim*}{Claim}
\newtheorem*{lemma*}{Lemma}

\newcommand{\overbar}[1]{\mkern 2mu\overline{\mkern-2mu#1\mkern-1mu}\mkern 1mu}

\def\Q{\ensuremath{\mathcal{Q}}}

\def\C{\ensuremath{\mathcal{C}}}
\def\T{\ensuremath{\mathcal{T}}}
\def\P{\ensuremath{\mathcal{P}}}
\def\cP{\ensuremath{\mathcal{P}}}
\def\Q{\ensuremath{\mathcal{Q}}}

\def\T{\ensuremath{\mathcal{T}}}

\newcommand{\ex}{ {\rm ex}}
\newcommand{\cl}{ {\rm cl}}
\newcommand{\cov}{ {\rm cov}}
\newcommand{\pp}{ {\rm pp}}

\usepackage{authblk}

\title{Clumsy packings of graphs}

\author[1]{Maria Axenovich}
\author[1]{ Anika Kaufmann }
\author[2]{Raphael Yuster}
\affil[1]{Department of Mathematics, Karlsruhe Institute of Technology, 76133 Karlsruhe, Germany}
 \affil[2]{Department of Mathematics, University of Haifa, Haifa 31905, Israel}
\begin{document}

\maketitle

\begin{abstract}
Let $G$ and $H$ be graphs. We say that $\cP$ is an {\em $H$-packing} of $G$ if $\cP$ is a set of edge-disjoint copies of $H$ in $G$. 
An $H$-packing $\cP$ is {\em maximal} if there is no other $H$-packing of $G$ that properly contains $\cP$.
Packings of maximum cardinality have been studied intensively, with several recent breakthrough results.
Here, we consider minimum cardinality maximal packings. An $H$-packing $\cP$ is {\it clumsy} if  it is maximal of minimum size.
Let $\cl(G,H)$ be the size of a clumsy $H$-packing of $G$.
We provide nontrivial bounds for $\cl(G,H)$, and in many cases asymptotically determine $\cl(G,H)$ for some generic classes of graphs $G$
such as $K_n$ (the complete graph), $Q_n$ (the cube graph),  as well as  square, triangular, and hexagonal grids.
We asymptotically determine $\cl(K_n,H)$ for every fixed non-empty graph $H$.
In particular, we prove that 
$$
\cl(K_n, H) = \frac{\binom{n}{2}- \ex(n,H)}{|E(H)|}+o(\ex(n,H)),
$$
where $\ex(n,H)$ is the extremal number of $H$.

A related natural parameter is $\cov(G,H)$, that is the smallest  number of copies of $H$ in $G$    (not necessarily edge-disjoint)  whose removal from $G$ results in an $H$-free graph.
While clearly $\cov(G,H) \le \cl(G,H)$, all of our lower bounds for $\cl(G,H)$ apply to $\cov(G,H)$ as well.
\end{abstract}

\section{Introduction}\label{sec:intro}

Let $G$ and $H$ be graphs. We say that $\cP$ is an {\em $H$-packing} of $G$ if $\cP$ is a set of edge-disjoint subgraphs of $G$ each isomorphic to $H$.
We shall refer to subgraphs isomorphic to $H$ as {\it copies} of $H$.
An $H$-packing $\cP$ is {\em maximal} if there is no other $H$-packing in $G$ that properly contains $\cP$.
An $H$-packing $\cP$ of $G$ is {\em perfect} if every edge of $G$ belongs to an element of $\cP$.
In case when there is a perfect $H$-packing of $G$, we say that $H$ {\it decomposes} $G$.
The case of perfect $H$-packings has been extensively studied. A seminal result of Wilson \cite{wilson-1975} asserts that perfect $H$-packings of $K_n$ always exist when certain (obviously necessary) divisibility conditions hold. Wilson's result has been extended from $K_n$ to graphs with a
sufficiently large minimum degree starting with Gustavsson \cite{gustavsson-1991} and culminating in results of Keevash \cite{keevash-2014} and Glock et al. \cite{GKLO-2016}.
In general, however, deciding if $G$ has a perfect $H$-packing, or computing $\pp(G,H)$, the maximum cardinality of an $H$-packing of $G$,
in NP-hard for every connected $H$ other than $K_2$ \cite{DT-1992}.

Here, we consider minimum cardinality maximal packings. An $H$-packing $\cP$ is {\it clumsy} if  it is maximal of minimum size.
Let $\cl(G, H)$ be the size of a clumsy $H$-packing of $G$. Thus, we are interested in the smallest {\em covering} of all the copies of $H$ in $G$ with
edge-disjoint copies of $H$. Let $\cov(G,H)$ be the smallest cardinality of a set of (not necessarily edge-disjoint) copies of $H$ in $G$
such that each copy of $H$ in $G$ has an edge in some element of the set. Thus, we have $\pp(G,H) \ge \cl(G,H) \ge \cov(G,H)$.
The notion of a clumsy packing was introduced by Gy\'arf\'as  et al. \cite{GLT-1988} for domino packings. It was further extended to general polyominoes by Walzer et al. \cite{WAU-2014}.

Our main results provide upper and lower bounds, and in many cases asymptotically determine $\cl(G,H)$ and $\cov(G,H)$ for
major generic classes of graph $G$
such as $K_n$ (the complete graph), $Q_n$ (the cube graph), 
as well as  square, triangular, and hexagonal grids.
To state our results we need to recall some notation.
We assume that the graphs under consideration are non-empty unless otherwise stated.
We denote the number of edges of a graph $G$ by $||G||$. Note that $\pp(G) = ||G||/||H||$ if a perfect $H$-packing of $G$ exists.
The extremal number $\ex(G,H)$ is the largest number of edges in a subgraph of $G$ that contains no copy of $H$ and $\ex(n,H)=\ex(K_n,H)$.
Using these notations, a trivial  lower bound for $\cov(G,H)$, and thus for $\cl(G,H)$ is therefore
\begin{equation}\label{e:1}
\frac{||G||-\ex(G,H)}{||H||} \le  \cov(G,H) \le  \cl(G,H)
\end{equation}
since we must cover at least $||G||-\ex(G,H)$ edges of $G$ with copies of $H$.
Our first main result concerns the case $\cl(K_n,H)$.
\begin{theorem}\label{t:kn}
Let $H$ be any fixed non-empty graph. Then,
$$
\cl(K_n, H)  =\frac{||K_n||- \ex(n, H)}{||H||} + o(\ex(n,H))\;.
$$
If $H$ is complete or bipartite, then the $o(\ex(n,H))$ term is only linear in $n$ and if $H$ is a forest, the $o(\ex(n,H))$ term is a constant.
\end{theorem}

Note that by (\ref{e:1}), the same results hold for $\cov(G,H)$ instead of $\cl(G,H)$.
The proof of Theorem \ref{t:kn} appears in Section \ref{s:clique}.

The asymptotic expression  for $\ex(n, H)$, given by Erd\H{o}s-Stone Theorem  when $\chi(H)\geq 3$, allows us to obtain the asymptotic ratio between the sizes of clumsy and perfect packings.

\begin{corollary}
For a fixed $H$ with $\chi(H) \ge 3$,
$$
\lim_{n\rightarrow \infty} \frac{\cl(K_n, H)}{\pp(K_n, H)} = \frac{1}{\chi(H)-1}\;.
$$
\end{corollary}

Our next result concerns another well-studied generic graph, the hypercube $Q_n$. In this setting, it is most natural to
evaluate $\cl(Q_n,Q_d)$ where $d \ge 2$ is fixed. Unlike the case of Theorem \ref{t:kn} where the true asymptotic is determined
and the main goal is to keep the set of uncovered edges (which is small in some cases) as large as possible,
in the hypercube setting we are only able to obtain upper and lower bounds and these do not coincide.
\begin{theorem}\label{t:qn}
For a fixed integer $d$, $d\geq 2$,  
$$
\Omega \left(\frac{\log d}{d~ 2^d}\right)  \le
\liminf_{n \rightarrow \infty}\frac{||Q_d||}{||Q_n||} \cl(Q_n, Q_d)  \le
\limsup_{n \rightarrow \infty}\frac{||Q_d||}{||Q_n||} \cl(Q_n, Q_d) \le
\frac{\sqrt{2\pi}}{\sqrt{d}}(1+o_d(1))
$$
and for $d=2$, 
$$
0.3932 \le
\liminf_{n \rightarrow \infty}\frac{||Q_2||}{||Q_n||} \cl(Q_n, Q_2)  \le
\limsup_{n \rightarrow \infty}\frac{||Q_2||}{||Q_n||} \cl(Q_n, Q_2) \le
\frac{2}{3}\;.
$$
\end{theorem}
It is worth noting that Offner \cite{offner-2009} proved that $\frac{||H||}{||Q_n||} \pp(Q_n, H) = 1-o(1)$
for any fixed subgraph $H$ of a hypercube, so the ratios in Theorem \ref{t:qn} also serve as ratios between $\pp(Q_n, Q_d)$ and $\cl(Q_n, Q_d)$.
The proof of Theorem \ref{t:qn} appears in Section \ref{s:cube}.

Our third main result consists of constructions of clumsy packings of grid graphs corresponding to the regular tessellations of the plane.
There are only three regular tessellations of the plane: the triangular, the square, and the hexagonal, each corresponding to an infinite graph
whose vertices correspond to respective $k$-gons ($k \in \{3,4,6\}$), and edges correspond to pairs of $k$-gons sharing sides.
We construct clumsy $C_k$-packings for each of these graphs and compute the exact limit ratio of the covered edges in these packings
which turns out to be $2/(k+1)$. In particular,  for a square grid graph $Gr_n$, our results here imply that $\cl({Gr}_n,C_4) = \frac{n^2}{5}(1+o(1))$.
The main result here, stated as Theorem \ref{t:rk}, appears in Section \ref{s:tiling}.

Finally, it is interesting to point out that $\cl(G,H)$ is {\em not} a monotone graph parameter, which might explain some of the difficulties that arise in its
determination. Indeed, consider the graph $G'$ which is the union of $k+1$ graphs $F_0, \ldots, F_k$, each isomorphic to $C_k$ such that 
$F_i$ is edge-disjoint from $F_j$, $1\leq i<j\leq k$, and $F_0$ shares its $i^{\rm th}$ edge with $F_i$, when we arbitrarily label the edges of $F_0$ with $1, \ldots, k$.
Let $H=C_k$ and let $G$ be obtained from $G'$ by deleting an edge from $F_0$. Then we see that $\cl(G',H)=1$ as $\{F_0\}$ forms a clumsy packing.
On the other hand $\cl(G, H)=k-1$.

\section{Clumsy packings in $K_n$}\label{s:clique}

Prior to  proving  Theorem \ref{t:kn} we need to recall some facts and definitions.
For a graph $H$, let $gcd(H)$ denote the greatest common divisor of its vertex degrees.
A graph $G$ is called {\em $H$-divisible} if $gcd(H)$ divides $gcd(G)$ and $||H||$ divides $||G||$.
For example, $K_n$ is $K_m$-divisible if $n \equiv 1 \bmod m(m-1)$.
Clearly, a necessary condition for a perfect $H$-packing of $G$ is that $G$ is $H$-divisible.
A seminal result of Wilson  \cite{wilson-1975} asserts that for every fixed graph $H$, if $n$ sufficiently large and $K_n$ is $H$-divisible, then $K_n$ has a perfect $H$-packing.
Recall also that the Tur\'an graph $T(n,k)$ is the complete $k$-partite graph on $n$ vertices whose parts form an equitable partition
(so the size of each part is either $\lfloor n/k \rfloor$ or $\lceil n/k \rceil$).
Tur\'an's Theorem asserts that $\ex(n,K_{k+1})=||T(n,k)||$ and that $T(n,k)$ is the unique extremal $K_{k+1}$-free graph with $n$ vertices.

The lower bound on $\cl(K_n,H)$ in Theorem \ref{t:kn} follows from (\ref{e:1}).
To prove the upper bound, we split the proof of Theorem \ref{t:kn} into four parts, depending on the structure of $H$.
The parts correspond to the cases $H=K_m$, $\chi(H) \ge 3$, $H$ is bipartite but not a forest, and $H$ is a forest. While the first two cases are rather standard, the proofs of the bipartite and forest cases are more involved, especially since we  know neither the structure nor the asymptotic value of $\ex(n,H)$ in these cases.

\subsection{$H=K_m$}

\begin{lemma}\label{l:1}
For fixed $m \ge 2$, $\cl(K_n, K_m) \le (||K_n|| - \ex(n, K_m))/||K_m||+O(n)$.
Moreover, for $n$ sufficiently large, $\cl(K_n, K_m)=  (||K_n|| - \ex(n,K_m)) / ||K_m||$ if $n$ is divisible by $m-1$ and $K_{n/(m-1)}$ is $K_m$-divisible.
\end{lemma}
\begin{proof}
We assume $m \ge 3$ as the case $m=2$ trivially holds.
We shall construct a maximal $H$-packing with the desired number of copies of $K_m$.
Assume first that $n$ is divisible by $m-1$ and $K_{n/(m-1)}$ is $K_m$-divisible.
Thus, given $G=K_n$, the Tur\'an graph $T(n,m-1)$ is a spanning subgraph of $G$. Denote the partite sets of  $T(n, m-1)$  by $V_1,\ldots, V_{m-1}$.
If $n$ is sufficiently large, we can use Wilson's Theorem to find a perfect $K_m$-packing of each $V_i$. The union of these $m-1$ perfect packings
is a $K_m$-packing of size $(||K_n|| - \ex(n, K_m))/||K_m||$, as required. It is a maximal packing since each copy of $K_m$ 
in $K_n$ contains an edge induced by one of $V_i$'s.

Next assume that $n$ is not of the aforementioned form. Let $n' < n$ be the largest integer such that $m-1$ divides $n'$ and $K_{n'/(m-1)}$ is $K_m$-divisible.
For example, every $n'$ of the form $n' \equiv m-1 \bmod m(m-1)^2$ satisfies this conditions. Hence $n-n' < m(m-1)^2$.
Now let $G=K_n$ and let $G'=K_{n'}$ be a subgraph of $G$.
By the previous paragraph, if $n$ is sufficiently large (and thus $n'$ is sufficiently large) there is a maximal $K_m$-packing $\cP$ of $G'$ of size
$(||K_{n'}|| - \ex(n', K_m))/||K_m|| \le (||K_n|| - \ex(n, K_m))/||K_m||$.
However, there may now be $K_m$ copies of $G$ that are not covered by $\cP$. Each such copy must therefore contain an edge incident to one of the $n-n'$ vertices of
$V(G) \setminus V(G')$. As there are at most $(n-n')n < m(m-1)^2n$ such edges, one can greedily add edge-disjoint copies of $K_m$ to $\cP$ to obtain a maximal packing of $G$ consisting of less than
$|\cP|+ m(m-1)^2n$ elements which is less than $(||K_n|| - \ex(n, K_m))/||K_m|| + m(m-1)^2n$.

One can improve the error term $m(m-1)^2n$ to a better linear term using a result of Caro and Yuster \cite{CY-1997}
asserting that for any sufficiently large $\ell$ and a fixed $m$, $K_\ell$ contains at least 
$$
\left\lfloor \frac{\ell}{m}\left\lfloor \frac{\ell-1}{m-1}\right\rfloor \right\rfloor -1
$$
pairwise edge-disjoint copies of $K_m$.
Thus there is a packing of $K_m$'s in $K_\ell$ covering at least 
$$
\frac{m(m-1)}{2} \left( \left\lfloor \frac{\ell}{m}\left\lfloor \frac{\ell-1}{m-1}\right\rfloor \right\rfloor -1 \right) \geq \frac{\ell(\ell-1)}{2}  - \frac{(m-2)\ell}{2} -  m(m-1)
$$
edges of $K_\ell$ where the latter inequality follows from 
\begin{eqnarray*}
\left( \left\lfloor \frac{\ell}{m}\left\lfloor \frac{\ell-1}{m-1}\right\rfloor \right\rfloor -1 \right) & \geq &  \frac{1}{m} \left(\ell \cdot\frac{\ell-m+1}{m-1}-m\right)-1\\
&=& \frac{1}{m} \left(\ell \cdot\frac{(\ell-1)+ (-m+2)}{m-1}-m\right)-1 \\
&=& \frac{\ell (\ell-1)}{m(m-1)} + \frac {\ell (-m+2)}{m(m-1)}-2\;.
\end{eqnarray*}

Let $G=K_n$ and $V(G) = V_1\cup \cdots \cup V_{(m-1)}$, where the parts form an equitable partition. Let $|V_i| = \ell_i$, $i=1, \ldots, m-1$.
Let $\cP_i$ be a densest $K_m$-packing of $G[V_i]$ with $K_m$ for $i=1, \ldots, m-1$.
Let $\cP= \cP_1\cup \cdots \cup \cP_{m-1}$.
We see that the set of edges covered by $\cP$ is included in the set of edges of the complement of the complete $(m-1)$-partite graph with parts $V_1, \ldots, V_{m-1}$, i.e., the Tur\'an graph $T(n,m-1)$. Thus $|\cP| \leq (||K_n|| - \ex (n, K_m))/||K_m||$.
Let the set  of edges in $G[V_i]$'s that are not covered by $\cP$ be denoted $E'$.  Then  
$$
|E'|\leq \sum_{i=1}^{(m-1)} \frac{\ell_i(\ell_i-1)}{2} - \left(\frac{\ell_i(\ell_i-1)}{2}  - \frac{(m-2)\ell_i}{2} -  m(m-1)\right)=
\frac{(m-2)n}{2} + m(m-1)^2\;.
$$
Thus we can greedily extend $\cP$ to a maximal $K_m$-packing of $G$ by adding at most $|E'|$ elements to cover each edge of $E'$ when possible. The resulting maximal packing will have size at most
$$
|\cP| + \frac{(m-2)n}{2} + m(m-1)^2  \leq  \frac{||K_n|| - \ex (K_n, K_m)}{||K_m||} + \frac{(m-2)n}{2} + m(m-1)^2\;.
$$
\end{proof}
It is worth noting that Lemma \ref{l:1} implies that for a fixed $m$,
$$
\lim_{n\rightarrow \infty} \frac{\cl(K_n, K_m)}{\pp(K_n, K_m)} = \frac{1}{m-1}\;.
$$

\subsection{$\chi(H) \ge 3$}

\begin{lemma}\label{l:2}
Let $H$ be a fixed graph with $\chi(H)=3$. Then, $\cl(K_n,H) \le (||K_n|| - \ex(n,H))/||H||+o(n^2)$. In particular, $\cl(K_n,H)\leq  (||K_n|| - \ex(n,H))/||H||+o(\ex(n,H))$.
\end{lemma}
\begin{proof}
Let $\chi(H) = r \ge 3$.
The Erd\H{o}s-Stone Theorem implies that $\ex(n,H) = ||T(n,r-1)||(1+o(1))$.
Let $G=K_n$ and let $G'=T(n,r-1)$ be a spanning subgraph of $G$. As in Lemma \ref{l:1}, we use Wilson's Theorem to find an $H$-packing $\cP$ of
the complement of $G'$ in $G$ (i.e. the vertex-disjoint cliques induced by the $r-1$ parts) which cover all but $O(n)$ edges of this complement.
Let $E^*$ denote the uncovered edges of the complement.
Notice that since $G'$ contains no copy of $H$, any copy of $H$ consisting only of edges of $G$ that not covered by $\cP$ must contain an edge from $E^*$.
We can thus extend $\cP$ to a maximal $H$-packing of $G$ using at most $|E^*|=O(n)$ additional elements.
The size of this maximal $H$-packing is therefore at most
$$
\frac{||K_n||- ||T(n,r-1)||}{||H||} + O(n) \le \frac{||K_n||- \ex(n,H)||}{||H||} + o(n^2)\;.
$$
\end{proof}
It is worth noting that Lemma \ref{l:2} implies that for a fixed $H$ with $\chi(H) \ge 3$,
$$
\lim_{n\rightarrow \infty} \frac{\cl(K_n, H)}{\pp(K_n, H)} = \frac{1}{\chi(H)-1}\;.
$$

\subsection{$\chi(H) =2$ and $H$ is not a forest}

For this case we prove the following lemma which, in turn, relies on several breakthrough results \cite{BKLO-2016,gustavsson-1991,keevash-2014} that imply that
for any fixed graph $H$, an $H$-divisible graph with sufficiently many vertices and sufficiently large minimum degree has a perfect $H$-packing.
\begin{lemma}\label{l:h-pack}
Let $H$ be a fixed graph. Then there exist $\delta=\delta(H) > 0$, $C=C(H)$, and  $N=N(H)$ such that the following holds.
If $G$ is a graph with $n > N$ vertices and minimum degree at least $(1-\delta)n$, then $G$ has an $H$-packing which covers all but at most
$Cn$ edges of $G$.
\end{lemma}

\begin{proof}
By any one of the results \cite{BKLO-2016,gustavsson-1991,keevash-2014}, for every fixed $H$ there exist $\epsilon=\epsilon(H) > 0$ and $N_1=N_1(H)$
such that any $H$-divisible graph $G$ with $n > N_1$ vertices and minimum degree at least $(1-\epsilon)n$ has a perfect $H$-packing. We can assume that $\epsilon<1/3$. 

Let $\delta= \epsilon/3$ and for notational convenience, let $gcd(H)=r$ and $||H||=h$.
Let $s$ be the smallest even integer larger than $6rh/\epsilon$. Let $N= \max\{N_1, \lceil 6s/\epsilon \rceil\}$.
Let $G$ be a graph with $n > N$ vertices and minimum degree $(1-\delta)n$. If $G$ were $H$-divisible, we would be done as $G$ would have a perfect
$H$-packing. Unfortunately, this might not be the case. 

Let $V=V(G)=\{v_1,\ldots,v_n\}$. Consider a set $S$ of new vertices, $V\cap S=\emptyset$, $|S|=s$, $s$ is even, $s> 6rh/\epsilon$.
We shall construct a new graph on a vertex set $V\cup S$ so that $V$ induces $G$ and so that the new graph is $H$-divisible. More specifically, we shall construct this new graph in such a way that all its vertex degrees are divisible by $2rh$. Then clearly each degree is divisible by $r$ and the number of edges is divisible by $h$.

We shall define a graph  $G'$  whose vertex set is $V \cup S$,  ~ $G'[V]=G$, ~ $G'[S]=K_s$, and  
 the adjacencies between $S$ and $V$ are defined by the following procedure.

We define these adjacencies in $n$ steps where initially before the first step, we take all $ns$ possible edges between $S$ and $V$ and in each step we delete a few of them.
Let $d_i$ denote the degree of $v$ in $G'$ before the first step (so $d_i=deg_G(v_i)+s$). Let $b_i \equiv {d_i} \pmod {2rh}$ so that $0 \le b_i < 2rh$.
In the first step we arbitrarily remove $b_1$ edges between $v_1$ and $S$. So, after this removal, the degree of $v_1$ becomes
$d_1-b_1 \equiv 0 \pmod {2r}$, and some vertices of $S$ (that is, precisely $b_1$ vertices of $S$) have degree  equal to $n+s-2$ while the other
$s-b_1$ vertices of $S$ have degree equal to $n+s-1$. In a general step $i$, we remove $b_i$ edges between $v_i$  and $S$
so that the $b_i$ endpoints of these edges in $S$ are the ones that presently have the highest degree.
Notice that at any point in this process, the degrees  of any two vertices of $S$ differ by at most $1$.
After this process ends, the resulting $G'$ has the following property. The degree of each $v_i\in V$ in $G'$ is $d_i-b_i \equiv 0 \pmod {2rh}$,
and for some $q$, each  vertex from   $S$ has degree either $q$ or $q+1$ in $G'$.

Let us next estimate $q$. The total number of edges removed in the aforementioned process is $\sum_{i=1}^n b_i < 2rhn$.
Thus, the number of non-neighbors of each vertex of $S$ is at most $\lceil 2rhn/s \rceil$ implying that $q \ge n+s-1-\lceil 2rhn/s \rceil$.
Let $S_1$ be the set of vertices of $S$ with degree $q+1$ in $G'$ and let $S_0$ be the set of vertices of $S$
with degree $q$ in $G'$. Recall that $|S_0|+|S_1|=s$ is even. We claim that both $|S_0|$ and $|S_1|$ are even. Assume otherwise, then both are odd.
But since all the degrees of all $v \in V$ in $G'$ are $0 \pmod {2rh}$ and in particular even, we have that $G'$ has an odd number of vertices with odd
degree, a contradiction. So, we have that both $S_1$ and $S_0$ are even. Take an arbitrary perfect matching in $G'[S_1]$ (recall, $G'[S]=K_s$ so this can be trivially
done) and remove it. Thus in the resulting new graph $G''$ all the degrees of the vertices of $S$ are precisely $q$ and we have not changed the degrees of
the other vertices. Now, suppose $q \equiv t \pmod {2rh}$ where $0 \le t < 2rh$. Take $t$ pairwise edge-disjoint perfect matchings of $S$ in $G''$
and remove them from $G''$ (this can easily be done greedily since after removing each perfect matching the minimum degree the subgraph induced by $S$ is larger than $s/2$
since $s > 4rh$). The resulting graph $G^*$ now has all of its degrees $0 \pmod {2rh}$, so $G^*$ is divisible by $H$.


Let us next estimate the minimum degree of $G^*$ which has $n+s$ vertices.
The degree in $G^*$ of every vertex $v_i \in V$ is $deg_G(v_i)+s-b_i \ge deg_G(v_i) \ge (1-\delta)n \ge (1-\epsilon/2)(n+s)$.
The degree of every vertex of $S$ in $G^*$ is $q-t \ge n+s-1-\lceil 2rhn/s \rceil - 2rh \ge (1-\epsilon/2)(n+s)$,
where we have used here that $s \ge 6rh/\epsilon$.
Therefore, $G^{*}$ has a perfect $H$-packing $\cP$. The elements of $\cP$ that are not entirely contained in $G$ are those that have an edge incident to $S$.  The number of such copies of $H$ is at most $s(n+s)$. 
Deleting the edges of these copies gives an $H$-packing of $G$ that covers all but at most $(n+s)s||H||  \le Cn$ edges for an appropriate constant $C$.
\end{proof}

\begin{lemma}\label{l:3}
Let $H$ be a bipartite graph that contains a cycle. Then, $\cl(K_n,H) \le (||K_n|| - \ex(n,H))/||H||+O(n)$.
In particular, $\cl(K_n,H)= (||K_n|| - \ex(n,H))/||H||+o(\ex(n,H))$.
\end{lemma}
\begin{proof}
Suppose that $H$ contains $k$ vertices (so $k \ge 4$) and let $2\ell$ denote the length of a shortest (hence even) cycle in $H$, $k\geq \ell$. On the one hand, any graph that is $C_{2\ell}$-free is also $H$-free and on the other hand, any graph that contains $K_{k,k}$ also contains $H$.
By the known lower bounds for $\ex(n,C_{2\ell})$ 
\cite{LPS-1988,margulis-1988}, see also the improved bounds in \cite{LUW-1999},
we have that $\ex(n,H) \geq \ex(n, C_{2\ell})= \Omega(n^{1+\frac{2}{3\ell+3}}) \ge \Omega(n^{1+\frac{4}{3k+6}})$.

Let $\delta,N,C$ be the constants from Lemma \ref{l:h-pack}, $\delta<1$,  and let $\gamma = \delta/2$.
Let $G=K_n$ where $n > N+k(4/\gamma)^k$ and let $G'$ be a spanning subgraph of $G$ which is $H$-free and has $||G'||=\ex(n,H)$. Recall that $\ex(n, H) \leq \ex(n, K_{k,k}) \leq kn^{2-1/k}$
as follows from Zarankiewicz' argument.
Let $L$ be the set of vertices of $G'$ whose degree in $G'$ is at least $\gamma n$ and let $S=V(G') \setminus L$ be the remaining vertices.
Note that $|L|\leq \gamma n/2$ since otherwise $||G'|| \geq \gamma^2 n^2/4 > \ex(n, H)$. In particular, we have that $|S|\geq n/2$.

We claim that $|L| \le k(4/\gamma)^k$. For consider the bipartite graph $B$ whose parts are $L,S$ and that contains all the edges of $G'$ with one endpoint in $L$
and the other in $S$. Then $||B|| \ge  |L|(\gamma n - |L|) \geq |L|\gamma n/2$ and since $B$ is $K_{k,k}$-free, it follows from the Kov\'ari-S\'os-T\'uran Theorem \cite{KST-1954}
that
$$
|L|\frac{\gamma}{2} n \le (k-1)^{1/k}n|L|^{1-1/k}+k|L|,
$$
which implies that $|L| < k(4/\gamma)^k$.

Let $G^*$ be the complement of $G[S]$. So, $G^*$ has $|S|=n-|L| \ge n - k(4/\gamma)^k > N$ vertices and its minimum degree is at least
$n-\gamma n - |L| \ge |S|(1-\delta)$. By Lemma \ref{l:h-pack}, $G^*$ contains an $H$-packing $\cP$ that covers all but at most $C|S| \le Cn$ edges of
$G^*$. Since $G'$ is $H$-free, any copy of $H$ in $G$ which does not have an edge in an element of $\cP$ must contain an edge that is either one of these
at most $Cn$ uncovered edges, or an edge incident to $L$. Hence we can augment $\cP$ to a maximal $H$-packing by adding at most
$Cn+|L|n$ elements to it.
\end{proof}

\subsection{$H$ is a forest}

We shall need additional results about special packings of trees in dense graphs.

\begin{lemma} \label{tree-packing} Let $k$ be a fixed positive integer  and $N$ be an integer,  $n> 6 k^2$.   Let $T$ be a forest on $k$ vertices. Then if $F$ is a graph on $n$ vertices with minimum degree at least $2n/3$ then $F$ contains $n$ edge-disjoint copies of $T$ such that for each vertex of $T$, its $n$ respective images in the $n$ copies  are distinct. In particular, each vertex of $F$ belongs to exactly $k$ copies of $T$.
\end{lemma}

\begin{proof}
We shall use induction on $k$ with a trivial basis $k=1$.  Let $T$ be a forest on $k$ vertices and $T'$ is the forest obtained from $T$ by removing a leaf $v$ adjacent to a vertex $v'$. Note that if there is no such leaf, then $T$ has no edges and the result follows trivially. Then by induction hypothesis there is a set $\T'$ of  $n$ pairwise edge-disjoint copies of $T'$ in $F$ such that in particular
the images of $v'$ are distinct. Let $F'$ be a graph obtained from $F$ by deleting the edges of copies of $T'$ from $\T'$  as well as by deleting all those edges of $F$ that 
join an image of $v'$ to the vertices of its copy of $T'$, for each copy of $T'$ from $\T'$.  We see that each vertex of $F$ is an image of each vertex of $T'$ in some copy of $T'$.
Thus the number of deleted edges that were incident to each vertex is at most $(k-1)^2 + (k-1)$.
Therefore the minimum degree of $F'$ is at least $2n/3 - (k-1)k\geq n/2$. By Dirac's theorem, we see that there is a Hamiltonian cycle in $F'$.
Extend each copy of $T'$ from $\T'$  to a copy of $T$ by picking a neighbor of the image of $v'$ in that copy on this cycle such that distinct vertices get distinct neighbors. These newly picked neighbors serve as images of $v$ in respective copies of $T$.
\end{proof}

\begin{lemma}\label{l:4}
Let $H$ be a nonempty forest. Then, $\cl(K_n,H) \le (||K_n|| - \ex(n,H))/||H||+O(1)$.
In particular, $\cl(K_n,H)= (||K_n|| - \ex(n,H))/||H||+o(\ex(n,H))$.
\end{lemma}
\begin{proof}
Assume that the number of vertices of $H$ is $k \ge 3$ and $H$ has at least two edges.
If $H$  has one edge, and $n\geq |V(H)|$, we have ${\rm ex}(n, H)=0$ and $\cl(K_n, H)=||K_n|| =(||K_n||-{\rm ex}(n, H))/||H||$.
Let $G=K_n$ and let $G'$ be a spanning subgraph of $G$ which is $H$-free and satisfies $||G'||=\ex(n,H)$.
By the simple bounds on $\ex(n,H)$ we have that $\lfloor n/2 \rfloor  \leq ||G'|| \le kn$.
Let $L$ be the set of vertices of $G'$ of degree at least $n/4$ and let $S$ be the set of remaining vertices, let $|L|=\ell$.
Since $(n/4)|L| \le 2||G'|| \le 2kn$ we have that $|L| \le 8k$.
It follows that $\delta(\overline{G'[S]}) \ge n-1-n/4-|L| \ge 2n/3$ (assuming $n \ge 12(8k+1)$) .
Unlike the previous cases, we shall first completely cover the edges of $\overline{G'}$ between $L$ and $S$, 
then we shall (almost completely) pack the remaining edges of $\overline{G'[S]}$.

Now, consider $H$, our given forest, let $H'$ be obtained from $H$ by deleting a leaf $u'$ adjacent to 
a vertex $u$. Let $H^*$ be a union of $\ell$ copies of $H'$ that pairwise share only a vertex corresponding to $u$. Apply Lemma \ref{tree-packing} to find $|S|$ edge-disjoint copies of $H^*$ in $F= \overbar{G[S]}$  such that each vertex of $S$ serves as an image of $u$ in some copy of $H^*$.
For a vertex $s\in S$, let $H^*_s$ be the copy of $H^*$ such that $s$ serves as an image of $u$.
Let $X_s$ be the set of all vertices in $L$ such that $s$ is not adjacent to its members, let  $d_s=|X_s|$,  note that $0\leq d_s\leq \ell=|L|$.
Delete $\ell -d_s$ copies of $H'$ from $H^*_s$ and add the pairs $\{sx: ~ x\in X_s\}$ to the edge set of the resulting graph, call it $H''_s$. Note that $H_s''$ is an edge disjoint union of $d_s$ copies of $H$  that covers all edges from $s$ to $L$ in the complement of $G'$.
Thus all graphs $H''_s$, for $s\in S$ are pairwise edge-disjoint and cover all the edges between $S$ and $L$ in the complement of $G'$. Let $\P = \{ H''_s: ~ s\in S\}$.
Let $R$ be a graph with a vertex set $S$  with remaining edges that do not belong to either  $G'$ or any of $H_s''$, $s\in S$. 
We see that the minimum degree of $R$ is at least $2n/3 - |H^*|^2 \geq 2n/3 - \ell^2k^2\geq |S|/2$ for $n$ sufficiently large.

By a result of Yuster  \cite{yuster-2000}, if $H$ is a tree and $|S|$ is sufficiently large, a graph with $|S|$ vertices and minimum degree at least $|S|/2$ has an $H$-packing where less than $||H|| < k$ edges remain uncovered.
The proof in \cite{yuster-2000} gives the same result when $H$ is a forest. 
Another way to see this is that  from any forest without isolated vertices, one can construct a tree consisting of two edge-disjoint copies of that forest, and pack with that tree thereby packing with the forest.

Let then $\Q$ be   an $H$-packing  of  $R$ where less than $k$  edges remain uncovered.

Hence $\P \cup \Q$ is an $H$-packing of $\overline{G'}$ which covers all but at at most
$k$ edges in $\overline{G'[S]}$ plus the edges connecting two elements of $L$ which are yet uncovered. However, there are at most $\binom{\ell}{2} \le 32k^2$  edges induced by $L$.
Hence, we can obtain a maximal $H$-packing of $G$ of size at most $|\cP \cup Q|+k+32k^2 \le (||K_n|| - \ex(n,H))/||H||+k+32k^2$, as required.
\end{proof}

\section{Hypercubes}\label{s:cube}

In this section we prove Theorem \ref{t:qn}. We first note that computing $\cl(Q_n, Q_2)$ (and moreover $\cl(Q_n,Q_d)$) is difficult already for very small values of $n$.
While $\cl(Q_3, Q_2)=2$, $\cl(Q_4, Q_2)=3$ are trivial, it is only known that $\cl(Q_5, Q_2)\in \{7,8\}$ \cite{kaufmann-2018}.
Recall also that $Q_n$ is an $n$-regular graph with $2^n$ vertices hence $||Q_n||=n2^{n-1}$. More generally, we observe the following.

\begin{lemma}\label{l:count}
The number of copies of $Q_d$ in $Q_n$ is $2^{n-d}\binom{n}{d}$.
Each edge of $Q_n$ belongs to $\binom{n-1}{d-1}$ copies of $Q_d$.
\end{lemma}

\begin{proof}
Each copy of $Q_d$ can be represented by an $n$-vector in $\{0,1,\star\}$ with $d$ entries of $\star$.
So, the first part of the lemma follows from the fact that there are $2^{n-d}\binom{n}{d}$ such vectors.
If an edge $e$ is fixed, its endpoints differ in exactly one position, say position $i$. Then the $i$'th coordinate corresponds to a $\star$
in any  copy of $Q_d$ containing $e$. There are $\binom{n-1}{d-1}$ ways to choose other $\star$ positions and the remaining coordinates must take the respective values of endpoints of $e$.
\end{proof}

Let $f(n,d) = ||Q_n||- ex(Q_n, Q_d)$ be the smallest size  of an edge subset $S$ of $Q_n$ such that each copy of $Q_d$ in $Q_n$ contains at least one element of $S$.
Identically, $f(n,d)$ is the transversal number of the hypergraph whose vertices are the edges of $Q_n$ and whose edges are the (edges of) the $Q_d$ copies in $Q_n$.
Let 
$$
c(d) = \lim_{n\rightarrow \infty} \frac{f(n,d)}{||Q_n||}\;.
$$
Alon, Krech, and Szab\'o \cite{AKS-2007} proved that for some absolute positive constant $C$,
\begin{equation}\label{e:alon}
\Omega \left( \frac{\log d }{ d~2^d} \right) \leq c(d) \leq  \frac{C}{d^2}\;.
\end{equation}

\subsection{Lower bound}

In this subsection we prove the simple lower bounds stated in Theorem \ref{t:qn}.
Let $\cP$ be a maximal $Q_d$-packing of $Q_n$.
Since $\cP$ is maximal, every copy of $Q_d$ in $Q_n$ contains an edge of a member of $\cP$. Hence, by Lemma \ref{l:count}, we are counting
$\binom{n}{d}2^{n-d}$ edges in  this way, but each edge may be counted many times, as it may appear in $\binom{n-1}{d-1}$ copies of $Q_d$.
Thus, the total number of edges of all elements of $\cP$ is at least $\binom{n}{d}2^{n-d}/\binom{n-1}{d-1}$.
Since each element of $\cP$ consists of $d2^{d-1}$ edges it follows that
$$
|\cP|  \ge \frac{\binom{n}{d}2^{n-d}}{\binom{n-1}{d-1}d2^{d-1}} = \frac{2^{n-2d+1}n}{d^2}\;.
$$
To improve this lower bound by a factor of $\log d $ we use (\ref{e:alon}). 
Indeed, if $\cP$ is the smallest possible set of $Q_d$'s in $Q_n$ that contains an edge of each $Q_d$ of $Q_n$ (namely $|\cP|=\cov(Q_n,Q_d)$),
then the set of all edges of members of $\cP$ forms a transversal of $Q_d$'s in $Q_n$.
By (\ref{e:alon}), $||Q_d||\cdot |\cP| \geq  \Omega \left( \frac{\log d }{ d~2^d} \right) ||Q_n||$, thus 
$$
\cl(Q_n, Q_d) \ge \cov(Q_n,Q_d) = |\cP| \geq \Omega \left( \frac{\log d }{ d~2^d} \right) \frac{||Q_n||}{||Q_d||}\;.
$$

To get a lower bound on $\cl(Q_n, Q_2)$, we use a result of Baber \cite{baber-2012} stating that $\ex(Q_n, Q_2) \leq 0.6068 ||Q_n||(1+o(1))$.
Thus by (\ref{e:1}), we have 
$$
\cl(Q_n,Q_2) \ge \frac{(||Q_n|| - \ex(Q_n,Q_2))}{||Q_2||} \geq  0.3932 \frac{||Q_n||}{||Q_2||}(1-o(1))\;.
$$
We note that Erd\H{o}s conjectured that $\ex(Q_n, Q_2) = \frac{1}{2}||Q_n||(1+o(1))$, so if true, the constant $0.3932$ in the last inequality can be replaced by $\frac{1}{2}$.

\subsection{Upper bound}

In this subsection we prove the upper bounds stated in Theorem \ref{t:qn}.
For $i=0, \ldots, n$ we denote by $V_i$ the set of vertices of $Q_n$ with $i$ one's in their vector representation. We say that vertices or respective vectors from $V_i$ have weight $i$.
For $i=1, \ldots, n$, let  $L_i$ be the set of edges of $Q_n$ with endpoints in $V_{i-1}\cup V_i$.  We call $L_i$ the $i^{\rm th}$ {\em edge layer} of $Q_n$. 
We provide constructions of maximal $Q_d$-packings $\cP$ of $Q_n$ such that the edges  of $\cP$ cover almost completely every $(d-1)^{\rm st}$ layer of $Q_n$,
for $d\geq 3$, and for $d=2$, these edges cover two out of every three consecutive layers of $Q_n$ almost completely.

Let $I=[n/2-\sqrt{n \log n},  n/2 + \sqrt{n \log n}]$. Observe that $\sum_{i\not\in I} |L_i| = o(||Q_n||)$, thus we shall focus on the middle layers $L_i$,  $i\in I$ and later
consider any maximal packing of the remaining layers. 
We denote the edge set in these {\it middle layers} by $M = \cup_{j\in I }L_j$.

We consider first the case $d=2$ and later see how to generalize our arguments to an arbitrary $d$.
Let $M= M_0\cup M_1\cup M_2$, where $e\in M_i$ if and only if $e\in L_j$, $j\equiv i \pmod 3$.

\begin{lemma}\label{l:claim1}
Let $j\in I$. There is a $Q_2$-packing denoted $\cP_j$, such that each member of $\cP_j$ contains at least one edge of $L_j \cup L_{j+1}$ and such that $\cP_j$ covers all but at most $O(n^{-1/3} (|L_j|+|L_{j+1}|))$ edges of $L_j \cup L_{j+1}$. 
\end{lemma}
\begin{proof}
Let $H_j$ be the hypergraph whose vertices correspond to the edges of $L_j \cup L_{j+1}$ and whose hyperedges are four-element subsets forming a copy of $Q_2$ having all of their edges in $L_j \cup L_{j+1}$. Since any two copies of $Q_2$ intersect in at most one edge, $H_j$ is simple (linear).
Note that the degree of an element of $L_j$ in $H_j$ is $n-j$ since it appears in precisely $n-j$ $Q_2$'s having all of their edges $L_j \cup L_{j+1}$.
Indeed, suppose this element is the edge $e=(u,v) \in L_j$ where $u$ is a vector of weight $j-1$  and $v$ is a vector of weight $j$.
Then a vertex $x$ of a $Q_2$ containing $e$ and which is adjacent to $u$ must be also of weight $j$ but distinct from $v$, so there are $n-j$ options to choose $x$.
The fourth vertex of this $Q_2$ is now completely determined.
Similarly, the degree of an element of $L_{j+1}$ in $H_j$ is $j$.
To see this, suppose this element is the edge $e=(u,v) \in L_{j+1}$ where $u$ is a vector of weight $j$ and $v$ is a vector of weight $j+1$.
Then a vertex $x$ of a $Q_2$ containing $e$ and which is adjacent to $v$ must be also of weight $j$ but distinct from $u$, so there are $j$ options to choose $x$.
The fourth vertex of this $Q_2$ is now completely determined.
We see, using that $j \in I$,  that the absolute difference between the degrees of any two vertices of $H_j$ is at most $|n-2j| \le 4\sqrt{n \log n}$.

A result of Alon, Kim, and Spencer \cite{AKS-1997}, implies that if a $4$-uniform hypergraph $H$  has minimum degree at least $D - O(\sqrt{D\log D})$, where $D$ is the maximum degree, then there is a matching in the hypergraph covering all but at most $|V(H)|O(D^{-1/3})$ vertices.
Since the maximum degree of $H_j$ is $\max\{j,n-j\} \ge n/2$,
we see that there is a matching of $H_j$ covering all but $(|L_j|+|L_{j+1}|)O(n^{-1/3})$ vertices.
This matching corresponds to a $Q_2$-packing, call it $\cP_j$, whose elements cover all but $O(n^{-1/3}(|L_j|+|L_{j+1}|))$ edges of $L_j\cup L_{j+1}$. 
\end{proof}

Let $\cP'=\cup_{j\in I,~ j\equiv 0 \pmod 3} \cP_j$. It is a $Q_2$-packing that covers all but $o(|M_0 \cup M_1|)$ edges of $M_0\cup M_1$, 
and does not cover any edge from $M_2$. Let $F$ denote the $o(|M_0 \cup M_1|)$ uncovered edges of $M_0\cup M_1$.
Now augment $\cP'$ to a maximal $Q_2$-packing $\cP$ of $Q_n$. 
We claim that each element of $\cP \setminus \cP'$ contains an edge from $F \cup (E(Q_n) \setminus M)$.
Indeed this just follows from the obvious fact that each $Q_2$ contains edges from  precisely two consecutive layers, hence
each $Q_2$ in $\cP \setminus \cP'$ has an edge which is not from $M_2$, thus from $F \cup (E(Q_n) \setminus M)$.

But now, since $|F \cup (E(Q_n) \setminus M)| = o(|M_0 \cup M_1|)+o(||Q_n||) = o(||Q_n||)$, it follows that
$$
\cl(Q_n,Q_2) \le |\cP| = \frac{2}{3}|M| + o(||Q_n||) \le \frac{2}{3}\frac{||Q_n||}{||Q_2||}(1+o(1))\;.
$$

Next consider the case $d \geq 3$.
We shall apply a similar idea as in the case $d=2$, by first finding a packing $\cP'$ of the middle layers with $Q_d$. 
Let $M_0$ be the union of $L_j$'s such that $L_j\subseteq M$ and $j\equiv 0 \pmod d$.
First we find a packing $\cP'$ such that $M_0$  is covered almost completely,  then we augment this packing with a few copies of $Q_d$ so that the resulting packing is
maximal. In  what follows we assume that $d$ is odd. For $d$ even the argument is very similar.
Notice that each $Q_d$ has edges from precisely $d$ consecutive layers, so when $d$ is odd, the {\em middle layer} of a $Q_d$ is well-defined.

\begin{lemma}\label{l:claim2}
Let $j\in I$ and $d\geq 3$. There is a $Q_d$-packing denoted $\cP_j$, such that each member of $\cP_j$ contains at least one edge of $L_j$ in its middle layer and such that $\cP_j$ covers all but at most $o(|L_j|)$ edges of $L_j$.
\end{lemma}
\begin{proof}
Let $H_j$ be the hypergraph whose vertices are the edges of $L_i$,  $i\in J$, where 
$J= [j- (d-1)/2 , j+ (d-1)/2 ]$ and whose hyperedges are $d2^{d-1}$-element subsets forming a copy of $Q_d$. 
We see that $H_j$ is $r=d2^{d-1}$-uniform and by symmetry, all vertices from the same layer $L_i$ have the same degree.
Let the maximum degree of $H_j$ be $D$, and denote the degree of a vertex from $L_i$ in $H_j$ by $d_i$.
Observe also that $D=d_j$.

We shall construct an $r$-uniform hypergraph $H'$ containing $H_j$ as a spanning subhypergraph such that $H'$ is almost regular, i.e., has degrees $D$ or $D-1$ and such that $E(H')-E(H_j)$ forms a simple hypergraph.
Let $H' = H_j \bigcup_{i\in J }H_i'$, where $H_i'$ is a simple $r$-uniform hypergraph satisfying $V(H_i')= L_i$, and all of the degrees of $H_i'$ are either $D-d_i$ or $D-d_i-1$.
Note that since $d_j = D$ we have that $H_j'$ is an empty hypergraph.
A result of Bollob\'as \cite{bollobas-1980} asserts that such an $H_i'$ exists if
1) $x(D-d_i)+y(D-d_i -1)$ is divisible by $r$, where $x$ and $y$ are the numbers of vertices of degree $D-d_i$ and $D-d_i-1$, respectively, 
2) $|E(H_i')|$ approaches infinity as $|V(H_i)|$ approaches infinity. The second condition is clearly satisfied when $D>d_i$.
To see that the first condition is satisfied for some $x$ and $y$, $x+y= |V(H_i')|$, 
observe that $x(D-d_i)+ y(D-d_i -1) = (D-d_i)|V(H_i')|-y$, so we can choose $y$ to be an integer between $0$ and $r$ such that 
$(D-d_i)|V(H_i')|-y$ is divisible by $r$.
We have that $|V(H_i')| = |L_i| = \binom{n}{i}i$. So, $H'$ is a hypergraph whose vertices have degrees $D$ or $D-1$ and whose co-degree is at most the co-degree of $H_j$. 

Next we shall compare the degree $D$ of $H'$ and its maximum co-degree $coD$.
We shall view the vertices of $Q_n$ as subsets of $[n]$. For a vertex $x \subseteq [n]$, let $Up(x)$ and $Down(x)$ be the up-set and down-set of $x$, respectively, i.e., 
the set of all supersets of $x$ and the set of all subsets of $x$.  Let $V_k$ be the $k^{th}$ vertex layer of $Q_n$, $k=0, \ldots, n$.

The degree of a vertex  $e=xy$, $x\subseteq y$ in $H_j$ corresponds to the number of copies of $Q_d$'s containing $e$ and having middle layer in $L_j$.
The number of ways to choose the maximal element of such a $Q_d$ is equal to $u= |Up(y)\cap V_{j+(d-1)/2}| \ge c n^{k}$.
The number of ways to choose the minimal element of such a $Q_d$ is equal to $d= |Down(x)\cap V_{j-(d-1)/2-1}| \ge c' n^ {d-1-k}$, 
where $c, c'$ are constants depending on $d$ and $k$ is the distance in $Q_n$ between the vertex-layer containing $y$ and the vertex layer $V_{j+(d-1)/2}$.
Then the degree of $e$ in $H_j$ is at least $cc' n^{d-1}$. 

Now we upper bound the co-degree of two vertices of $H_j$:  $xy, x'y'$,  $x\subseteq y$ and $x'\subseteq y'$. We therefore need to find the  number of copies of $Q_d$ containing $xy$ and $x'y'$  and having middle layer in $L_j$. The number of ways to choose the maximal element of such a $Q_d$ is equal to $|Up(y\cup y')\cap V_{j+(d-1)/2}| \le c'' n^{k}$, where $k$ is the distance in $Q_n$ between the vertex-layer containing $y \cup y'$ and the vertex layer $V_{j+(d-1)/2}$.
The number of ways to choose the minimal element of such a $Q_d$ is equal to $|Down(x\cap x') \cap V_{j-(d-1)/2-1}| \le c'''n^{k'}$, 
where $k'$ is is the distance in $Q_n$ between the vertex-layer containing $x\cap x'$ and the vertex layer $V_{j-(d-1)/2-1}$.
Then the co-degree of $xy$ and $x'y'$ is at most $c''c''' n^{k+k'}$. We see that $x\cap x' \subseteq  y\cup y'$ and $|(y\cup y')-(x\cap x')|\ge 2$, 
so $k'\leq d-2 -k$. Thus  $coD \leq Cn^{d-2}$, for a constant $C$.
Since any two hyperedges from $E(H')-E(H_j)$ intersect in at most one vertex, the maximum co-degree of $H'$ is also at most $Cn^{d-2}$.

We need a result of Frankl and R\"odl \cite{FR-1985} on near perfect matchings of uniform hypergraphs.
They have proved that  for an integer $r \geq 2$ and a real $\beta > 0$ there exists $\mu=\mu(r,\beta) > 0$ such that
if the $r$-uniform hypergraph $L$ has the following properties for some $t$:
(i) The degree of each vertex is between $(1-\mu)t$ and $(1+\mu)t$,
(ii) the maximum co-degree is at most $ \mu t$, then $L$ has a matching of size at least $(|V(L)|/r)(1-\beta)$.
Applying their result to our hypergraph $H'$ (which is almost regular) 
we obtain that it has a matching $M$ that covers all but $o(|V(H')|) = o(|L_j|)$ vertices of $H'$.
In particular, we see that $M$ covers all, but at most $o(|L_j|)$ vertices from $L_j$. Since all edges from $H'$ that contain vertices from $L_j$ are all from $H_j$, we see that the set $M'$ of  hyperedges of $M$ containing vertices from $L_j$ corresponds to pairwise edge-disjoint copies of $Q_d$ with middle layer in $L_j$.
These copies cover all but $o(|L_j|)$ edges of $L_j$. 
Let $\cP_j$ correspond to the hyperedges of $M'$. We have that $|\cP_j| = (|L_j|/m_d)(1+o(1))$, where $m_d$ is the number of edges in the middle layer of $Q_d$:
$$m_d=\binom{d}{(d-1)/2} (d-1)/2 \approx (d/2) 2^d/\sqrt{\pi d/2} = 2^{d-1/2} \sqrt{d}/\sqrt{\pi}.$$
\end{proof}

The rest of the construction is done as in the case $d=2$. 
Let $\cP' = \cup_{j\in I, j \equiv 0 \pmod d}  \cP_j$.
Let $M= M_0\cup \cdots \cup M_{d-1}$, where $e\in M_i$ if and only if $e\in L_j$, $j \equiv i \pmod d$.
Thus, $\cP'$ covers all but $o(|M_0|)$ edges of $M_0$.
Let $F$ denote these $o(|M_0|)$ uncovered edges of $M_0$.
Now augment $\cP'$ to a maximal $Q_d$-packing $\cP$ of $Q_n$. 
We claim that each element of $\cP \setminus \cP'$ contains an edge from $F \cup (E(Q_n) \setminus M)$.
Indeed this just follows from the obvious fact that each $Q_d$ contains edges from precisely $d$ consecutive layers, hence
each $Q_d$ in $\cP \setminus \cP'$ has an edge which is not from $M_1 \cup \cdots \cup M_{d-1}$, thus from $F \cup (E(Q_n) \setminus M)$.

But now, since $|F \cup (E(Q_n) \setminus M)| = o(|M_0|)+o(||Q_n||) = o(||Q_n||)$, it follows that
$$
\cl(Q_n,Q_d) \le |\cP| = \frac{1}{d}\frac{||Q_n||}{m_d} + o(||Q_n||) = \frac{\sqrt{2 \pi} }{\sqrt{d}(1-o_d(1))}  \frac{||Q_n||}{||Q_d||}(1+o_n(1))\;.
$$

\qed

\section{Regular planar tilings}\label{s:tiling}

In this section we consider the regular tessellations (tilings) of the Euclidean plane.
It is well-known that there are only three such tilings. The triangular tiling $R_3$, the square tiling $R_4$, and the hexagonal (honeycomb) tiling $R_6$.
Viewed as infinite graphs, the vertices and edges of $R_k$ ($k=3,4,6$) are those of the regular $k$-gons comprising it.

To naturally define clumsy packing and perfect packing of $R_k$, we consider parametrized finite subgraphs of $R_k$.
Assume that the edges of $R_k$ have unit length and that there is an edge of $R_k$ connecting the origin $(0,0)$ and $(1,0)$. This uniquely defines all the Euclidean points of the vertices of $R_k$.
For an integer $n$, let $R_k(n)$ be the induced subgraph of $R_k$ on vertices inside $[0,n) \times [0,n)$.
So, for example $R_4(n)$ is just the square $n \times n$ grid $Gr_n$.
Let
$$
\cl(R_k) = \lim_{n \rightarrow \infty} \frac{k \cdot \cl(R_k(n),C_k)}{||R_k(n)||} \hspace{1cm} \pp(R_k) = \lim_{n \rightarrow \infty} \frac{k \cdot \pp(R_k(n),C_k)}{||R_k(n)||}\;.
$$
The fact that these limits exist will follow in particular from the proof below. So, in $\cl(R_k)$ and $\pp(R_k)$ we want to measure the ``fraction'' of edges of $R_k$
that are covered by the ``smallest'' (resp. ``largest'') maximal packing of $R_k$.
Note that it is straightforward that $R_3$ has a perfect triangle packing and that $R_4$ has a perfect $C_4$-packing hence $\pp(R_3)=\pp(R_4)=1$.
Clearly, $R_6$ does not have a perfect $C_6$-packing as it is $3$-regular, but it is a straightforward exercise to pack $R_6$ with $C_6$ such that the unpacked
edges form a perfect matching, hence $\pp(R_6)=2/3$. In the next theorem we determine $\cl(R_k)$.
\begin{theorem}\label{t:rk}
$\cl(R_k)=\frac{2}{k+1}$.
\end{theorem}
\begin{proof}
Consider first the case of $R_3$. The pattern on the right side of Figure \ref{f:triangle}
shows how to obtain a maximal triangle packing of $R_3$ where the ratio between covered and uncovered edges is $\frac{1}{2}$.
More formally, this pattern shows that $\frac{\cl(R_3(n),C_3)}{||R_3(n)||} \le \frac{1}{6}+o_n(1)$ implying that
$\limsup_{n \rightarrow \infty} \frac{3 \cdot \cl(R_3(n),C_3)}{||R_3(n)||} \le \frac{1}{2}$.
Consider the subgraph $H$ of $R_3$ shown on the left side of Figure \ref{f:triangle}.
Observe that $H$ has three internal edges and $6$ boundary edges.
Clearly, there is a covering $\C$ of $R_3$ with copies of $H$ such that the internal edges of each copy are pairwise edge-disjoint, while the boundary edges are shared between two copies in $\C$.
Consider some maximal $C_3$-packing $\cP$ of $R_3$ and consider some $H \in \C$.
We weigh the number of edges of $H$ covered by $\cP$ by giving each covered internal edge of
$H$ a weight $1$ and each covered boundary edge the weight $\frac{1}{2}$.
We claim that the weight of each $H \in \C$ is at least $3$. Indeed, if the internal triangle of $H$ is in $\cP$, we are done. Otherwise, at least one of the internal edges is covered, which means that there is a non-internal triangle of $H$ in $\cP$ which consists of one internal edge and two boundary edges. This already yields a weight of
$2$. But then the other two non-internal triangles of $H$ must intersect elements of $\P$ as well, so each gives an additional weight of at least $\frac{1}{2}$. Now, since the weight of each $H \in \C$ is at least $3$
and since its total weight to the edge count is $6$ as it has three internal edges and
six boundary edges (so $6 \times \frac{1}{2}+ 3 \times 1 = 6$), we have that
$\frac{\cl(R_3(n),C_3)}{||R_3(n)||} \ge \frac{1}{6}-o_n(1)$ implying that
$\liminf_{n \rightarrow \infty} \frac{3 \cdot \cl(R_3(n),C_3)}{||R_3(n)||} \ge \frac{1}{2}$.
\begin{figure}[h!]
\center
\includegraphics[scale=0.3, trim=100 150 100 0]{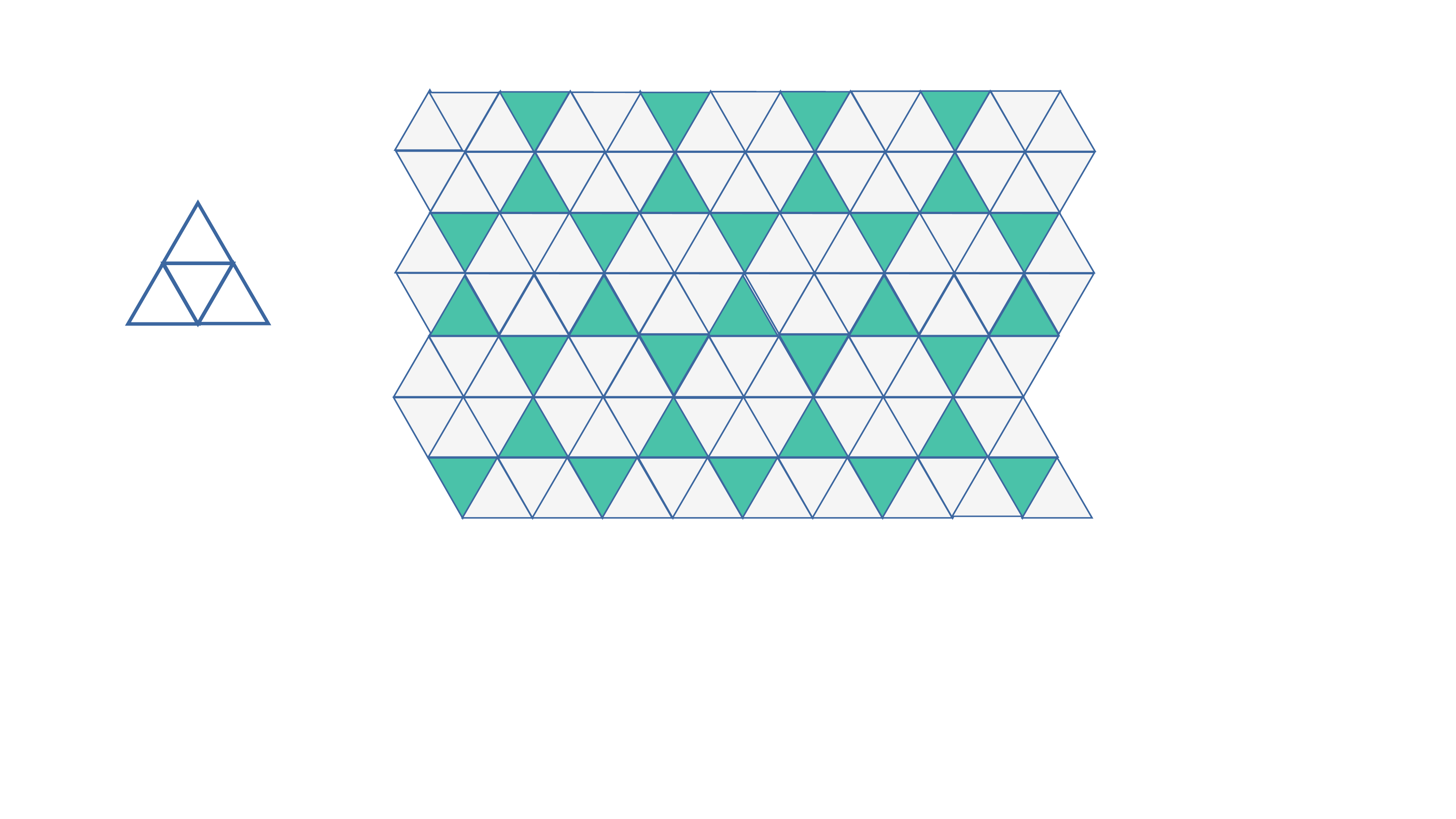}
\caption{A clumsy triangle packing of $R_3$ (right) and a gadget subgraph for the lower bound proof (left).}
\label{f:triangle}
\end{figure}

Consider next the case of $R_6$. The pattern on the right side of Figure \ref{f:hexagon}
shows how to obtain a maximal $C_6$-packing of $R_6$. Take the packing to consist of the internal $C_6$ of each colored region.
It is easy to verify that the ratio between covered and uncovered edges of this maximal packing is $2/7$.
More formally, this pattern shows that $\frac{\cl(R_6(n),C_6)}{||R_3(n)||} \le \frac{1}{21}+o_n(1)$ implying that
$\limsup_{n \rightarrow \infty} \frac{6 \cdot \cl(R_R(n),C_6)}{||R_3(n)||} \le \frac{2}{7}$.
Consider now the subgraph $H$ of $R_6$ shown on the left side of Figure \ref{f:hexagon}.
Observe that $H$ has $12$ internal edges and $18$ boundary edges.
As the left side of Figure \ref{f:hexagon} shows, there is a covering $\C$ of $R_6$ with copies of $H$ such that the internal edges of each copy are pairwise edge-disjoint, while the boundary edges are shared between two copies in $\C$.
Consider some maximal $C_6$-packing $\cP$ of $R_6$ and consider some $H \in \C$.
We weigh the number of edges of $H$ covered by $\cP$ by giving each covered internal edge of
$H$ a weight $1$ and each covered boundary edge the weight $\frac{1}{2}$.
We claim that the weight of each $H \in \C$ is at least $6$. Indeed, if the internal $C_6$ of $H$ is in $\cP$, we are done.
Otherwise, at least one of the internal edges of the internal $C_6$ is covered, which means that there is a non-internal $C_6$ of $H$, call it $X \in \cP$ which consists of three internal edges of $H$ and three boundary edges of $H$. This already yields a weight of
$4.5$. If $\cP$ contains an additional non-internal $C_6$ of $H$, then we get a weight of $9$ and we are done. Otherwise, the three non-internal $C_6$ of $H$ 
which are edge-disjoint from $X$ each contribute at least $\frac{1}{2}$ as $\cP$ is a maximal packing, hence overall weight at least $6$ as claimed.
Now, since the weight of each $H \in \C$ is at least $6$
and since its total weight to the edge count is $21$ as it has $12$ internal edges and
$18$ boundary edges (so $18 \times \frac{1}{2}+ 12 \times 1 = 21$), we have that
$\frac{\cl(R_6(n),C_6)}{||R_6(n)||} \ge \frac{1}{21}-o_n(1)$ implying that
$\liminf_{n \rightarrow \infty} \frac{6 \cdot \cl(R_6(n),C_6)}{||R_6(n)||} \ge \frac{2}{7}$.
\begin{figure}[h!]
\center
\includegraphics[scale=0.3, trim=100 150 100 0]{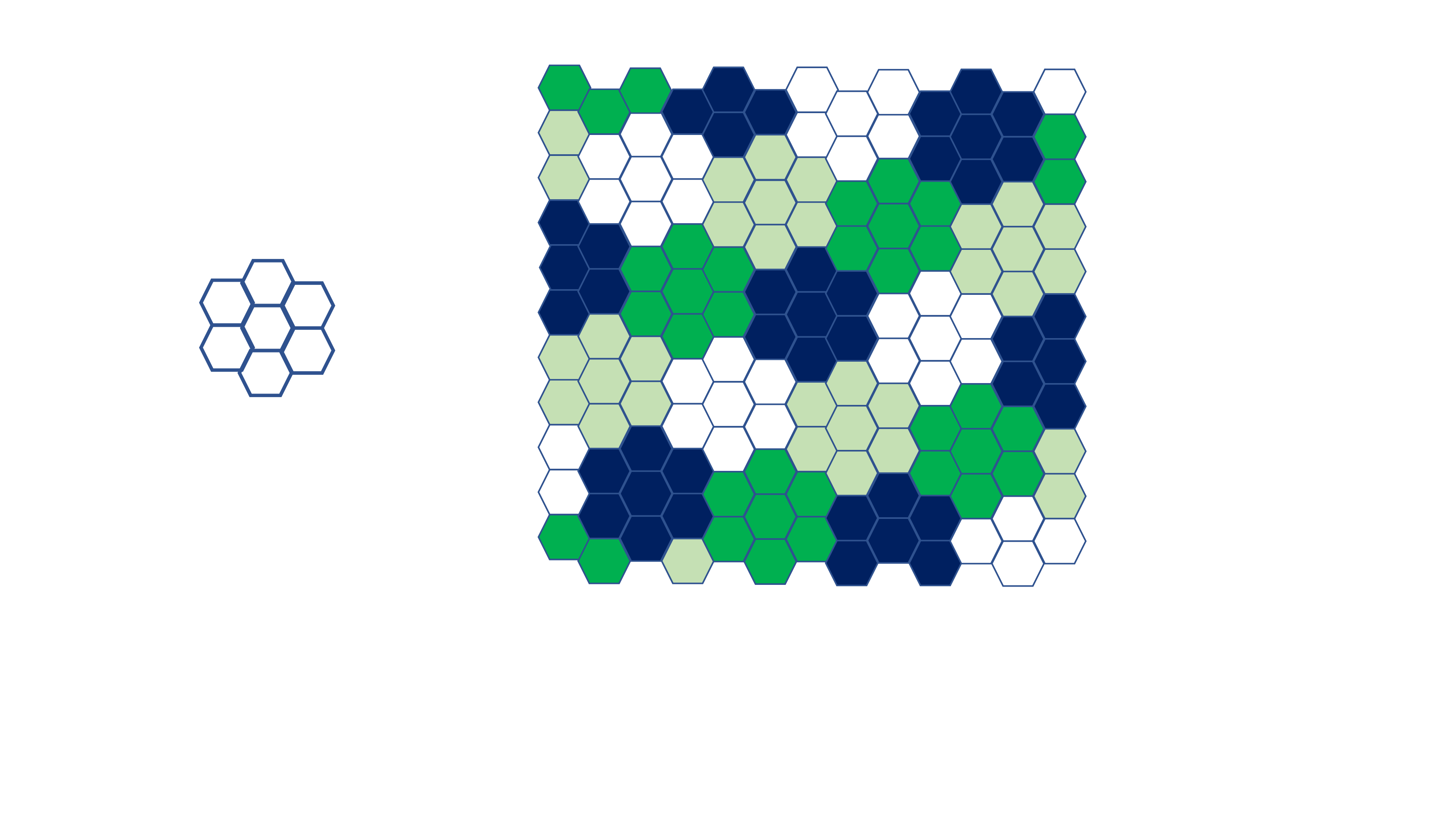}
\caption{A clumsy triangle packing of $R_6$ (right) and a gadget subgraph for the lower bound proof (left).}
\label{f:hexagon}
\end{figure}

Consider next the case of $R_4$. The idea in this case is close to the one on clumsy packings of polyominoes \cite{WAU-2014}.
The pattern on the right side of Figure \ref{f:square}
shows how to obtain a maximal $R_4(d)$-packing of $R_4$. In fact, the figure specifically shows the case $d=4$ but the generalization is obvious. Notice also that
$R_4(2)=C_4$.
Observing the proportion of the edges of the packing in each column and each row of $R_4$, we have
\begin{eqnarray*}
\cl(R_4(n), R_4(d)) & \leq & \frac{ d(d-1)} {2(d-1)(2d-2) + 1} \frac{||R_4(n)||}{||R_4(d)||}(1+o(1))\\
&= & \frac{ d^2-d}{ 4d^2 -8d +5} \frac{||R_4(n)||}{||R_4(d)||}(1+o(1)).
\end{eqnarray*}
For the lower bound, assume that $\cP$ is a maximal $R_4(d)$-packing of $R_4(n)$.
We see that each copy of $R_4(d)$ in $R_4(n)$ shares an edge with a copy of an element of $\cP$.
From the left side of Figure \ref{f:square} we see marked all the positions of the lower left corner of a copy of $R_4(d)$ that shares an edge with the marked copy
of $R_4(d)$. The number of such positions is $(2d-1)^2-4$. 
Therefore, if $x$ is the total number of $R_4(d)$'s in $R_4(n)$, then 
$x\leq ((2d-1)^2-4)\cP (1+o(1))$. Since $x=n^2(1-o(1))$, we have that 
\begin{eqnarray*}
\cl(R_4(n), R_4(d))& \geq & |\cP| \\
& \geq & \frac{n^2}{((2d-1)^2-4)} (1+o(1)) \\
&= &\frac{||R_4(n)||}{2((2d-1)^2-4)}  (1+o(1))\\
& = & \frac{(d-1)2d}{ 2((2d-1)^2 -4)} \frac{||R_4(n)||}{||R_4(d)||} (1+o(1))\\
&= & \frac{ d^2-d}{ 4d^2 - 4d -3} \frac{||R_4(n)||}{||R_4(n)||} (1+o(1)).
\end{eqnarray*}
Note that the upper and the lower bounds match for $d=2$, giving the claimed  value $\cl(R_4)=2/5$.
\begin{figure}[h!]
\begin{center}
\begin{tabular}{cc}
\includegraphics[scale=0.2, trim=0 0 0 2000]{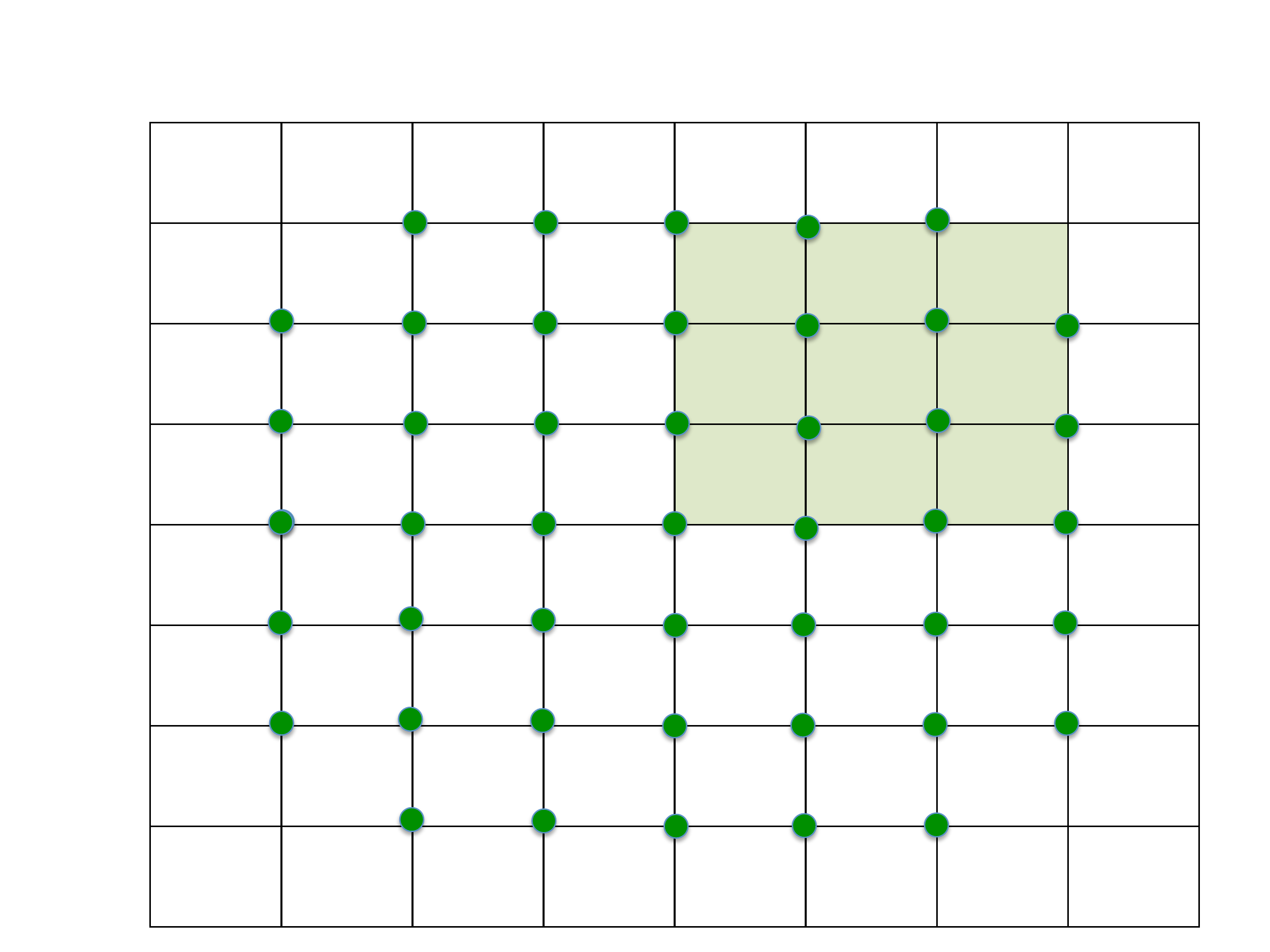} & \includegraphics[scale=0.35, trim=0 410 0 0]{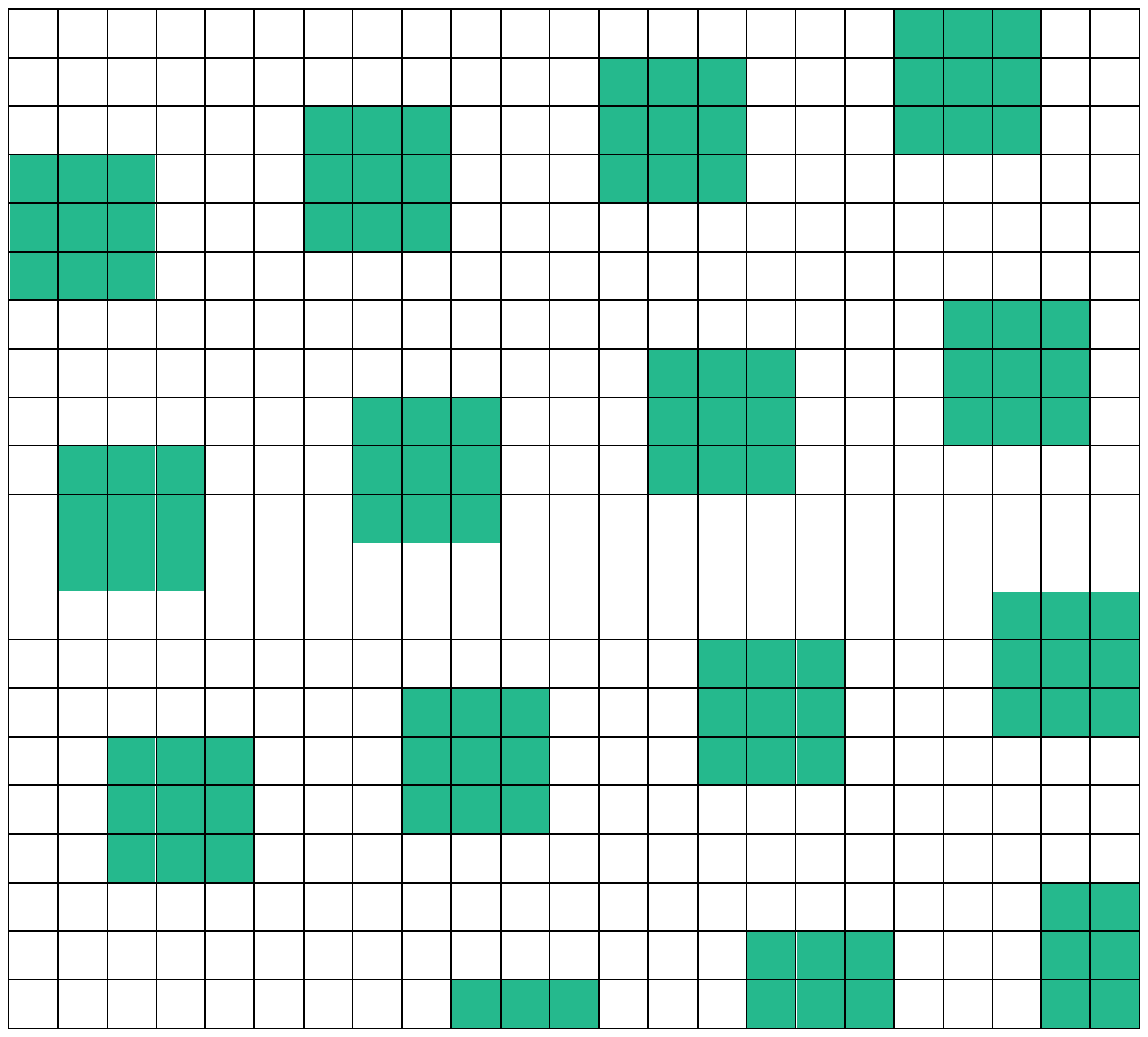}
\end{tabular}
\caption{Clumsy packing of $R_4(d)$ in $R_4$ (right) and the lower bound argument (left).}
\label{f:square}
\end{center}
\end{figure}
\end{proof}

\bibliographystyle{plain}

\bibliography{references}

\end{document}